\documentclass[11pt]{amsart}
\newcommand{\ben}{\begin{enumerate}}
\newcommand{\een}{\end{enumerate}}
\newcommand{\eq}[2][label]{\begin{equation}\label{#1}#2\end{equation}}
\newcommand{\av}[2]{\langle #1\rangle_{_{\scriptstyle #2}}}

\newcommand{\len}[2]{\big|[#1,#2]\big|}

\newcommand{\avm}[2]{\langle #1\rangle_{_{\scriptstyle #2,\mu}}}

\newcommand{\ve}{\varepsilon}
\usepackage{amsfonts}
\usepackage{amssymb}
\usepackage{amsmath}
\usepackage{graphicx,color,hyperref}

\newcommand{\bel}[1]{\boldsymbol{#1}}

\newcommand{\ma}{Monge--Amp\`{e}re }

\newcommand{\BMO}{{\rm BMO}}
\newcommand{\Oe}{\Omega_\varepsilon}

\newcommand{\T}{\mathcal{T}}

\newcommand{\rtde}{\sqrt{\delta^2-\varepsilon^2}}

\newcommand{\rn}{\mathbb{R}^n}

\newtheorem{theorem}{Theorem}[section]

\newtheorem{lemma}[theorem]{Lemma}
\newtheorem{corollary}[theorem]{Corollary}

\newtheorem*{theorem*}{Theorem}{\bf}{\it}
\newtheorem*{proposition*}{Proposition}{\bf}{\it}
\newtheorem*{observation*}{Observation}{\bf}{\it}
\newtheorem*{lemma*}{Lemma}{\bf}{\it}

\theoremstyle{definition}
\newtheorem{definition}[theorem]{Definition}

\theoremstyle{remark}
\newtheorem{remark}[theorem]{Remark}

\numberwithin{equation}{section}

\allowdisplaybreaks
\begin{document}

\title[Inequalities for BMO on {\large $\alpha$}-trees]{Inequalities for BMO on {\Large $\alpha$}-trees}

\author{Leonid Slavin}
\address{University of Cincinnati}
\email{leonid.slavin@uc.edu}

\author{Vasily Vasyunin}
\address{St. Petersburg Department of the V.~A.~Steklov
Mathematical Institute, RAS}
\email{vasyunin@pdmi.ras.ru}
\thanks{L. Slavin's research was supported in part by the National Science Foundation (DMS-1041763)}

\thanks{V. Vasyunin's research was supported by the Russian Science Foundation grant 14-41-00010}



%\date{Oct 26, 2011}

\begin{abstract}
We develop technical tools that enable the use of Bellman functions for BMO defined on $\alpha$-trees, which are structures that generalize dyadic lattices. As applications, we prove the integral John--Nirenberg inequality and an inequality relating $L^1$- and $L^2$-oscillations for BMO on $\alpha$-trees, with explicit constants. When the tree in question is the collection of all dyadic cubes in $\rn,$ the inequalities proved are sharp. We also reformulate the John--Nirenberg inequality for the continuous BMO in terms of special martingales generated by BMO functions. The tools presented can be used for any function class that corresponds to a non-convex Bellman domain. 
\end{abstract}
\maketitle
\section{Preliminaries and main results}
Let $\mathcal{D}$ stand for the collection of all open dyadic cubes in $\rn.$ This collection is uniquely defined by the choice of the root cube, say $Q_0=(0,1)^n.$ If a cube $Q$ is fixed, then $\mathcal{D}(Q)$ is the collection of all dyadic subcubes of $Q.$

By $\av{\varphi}J$ we denote the average of a locally integrable function over a set $J$ with respect to the Lebesgue measure; if a different measure, $\mu,$ is involved, we write
$\av{\varphi}{J,\mu}.$ Thus,
$$
\av{\varphi}J=\frac1{|J|}\int_J\varphi,\qquad \avm{\varphi}{J}=\frac1{\mu(J)}\int_J\varphi\,d\mu.
$$

For $p>0,$ and a function $\varphi\in L^p_{loc},$ let
$$
\Delta_{p,J}(\varphi)=\av{|\varphi-\av{\varphi}J|^p}J\quad\text{and}\quad 
\Delta_{p,\mu,J}(\varphi)=\avm{|\varphi-\avm{\varphi}J|^p}J.
$$ 
We will refer to both $\Delta_{p,J}$ and $\Delta_{p,\mu,J}$ as the $p$-oscillation of $\varphi$ over $J;$ this will not cause confusion, as the measure is always fixed. We will mainly need these definitions for $p=1$ and $p=2.$ Observe that
$$
\Delta_{2,\mu,J}(\varphi)=\avm{\varphi^2}J-\avm{\varphi}J^2.
$$
 
Let {\it the dyadic $\BMO$ on $\rn$} be defined by
\eq[1]{
\BMO^d(\rn)=\big\{\varphi\in L^1_{loc}\colon\|\varphi\|_{\BMO^d}:=\sup_{J\in\mathcal{D}}\big(\Delta_{2,J}\big)^{1/2}<\infty\big\}.
}
We will also use $\BMO^d(Q)$ when the supremum is taken over all $J\in\mathcal{D}(Q)$ for some cube $Q.$
For $\ve>0,$ the symbols $\BMO_\ve^d(\rn)$ and $\BMO_\ve^d(Q)$ will stand for the set of all $\BMO^d$ functions on the appropriate domain with norm not exceeding $\ve.$ 

Elements of $\BMO$ are locally exponentially integrable. The classical result that quantifies this property is the John--Nirenberg inequality~(\cite{jn}). Here we state it in the integral form.
\begin{theorem}[John, Nirenberg; integral form]
\label{jn}
There exists $\ve^d_0(n)>0$ such that for every $0\le\ve<\ve^d_0(n)$ there is a function
$C(\ve,n)>0$ such that for any $\varphi \in \BMO_{\ve}^d(\rn)$ and any $Q\in\mathcal{D}$
\eq[i2]{
\av{e^\varphi}Q\le C(\ve,n)e^{\av{\varphi}{\scriptscriptstyle Q}}.
}
\end{theorem}
Let us reserve the names $\ve^d_0(n)$ and $C(\ve,n)$ for the sharp values of these constants in~\eqref{i2}, i.e., the largest possible $\ve_0^d(n)$ and the smallest possible $C(\ve,n)$. (The constant $\ve_0^d(n)$ is of particular importance: it is easy to show that $\ve_0^d(n)$ is the supremum of $\ve>0$ such that for any $\varphi\in\BMO^d(\mathbb{R}^n),$ $e^{\ve\varphi/\|\varphi\|_{\BMO^d}}$ is a dyadic $A_2$ weight.) In~\cite{sv}, we computed these constants for $n=1:$
\begin{theorem}[\cite{sv}]
$$
\ve_0^d(1)=\sqrt{2}\log 2, \quad C(\ve,1)=\frac1{2\,e^{\frac\ve{\sqrt2}}-e^{\sqrt2\ve}}.
$$
\end{theorem}
In that paper, a family of Bellman functions for the John--Nirenberg inequality on the {\it continuous} (as opposed to dyadic) BMO was constructed, and an element of that family was identified as the Bellman function for the dyadic $\BMO^d(\mathbb{R}).$ Although it was intuitively clear at the time how to pick an element of the family that would work in dimensions greater than~1, the computation in~\cite{sv} was unsuitable for higher dimensions. In this paper, we develop several technical tools that allow us to prove the following theorem:
\begin{theorem}
\label{t1}
$$
\ve_0^d(n)=\frac{2^{n/2}}{2^n-1}\,n\,\log 2,\quad C(\ve,n)=
\frac{2^n-1}{2^n\, e^{2^{-n/2}\ve}-e^{2^{n/2}\ve}}.
$$
\end{theorem}

Theorem~\ref{t1} itself is a partial corollary of the corresponding result for special structures that generalize dyadic lattices, which we now define. 
\begin{definition}
\label{tree}
Let $(X,\mu)$ be a measure space with $0<\mu(X)<\infty.$ Let $\alpha\in(0,1/2].$ A collection $\T$ of measurable subsets of $X$ is called an $\alpha$-tree, if the following conditions are satisfied:
\ben
\item
$X\in\T.$
\item
For every $J\in\T,$ there exists a subset $C(J)\subset\T$ such that 
\ben
\item
$J=\bigcup_{I\in C(J)} I,$
\item
the elements of $C(J)$ are pairwise disjoint up to sets of measure zero,
\item
for any $I\in C(J),$ $\mu(I)\ge\alpha\mu(J).$
\een
\item
$\T=\bigcup_m\T_m,$ where $\T_0=\{X\}$ and $\T_{m+1}=\bigcup_{J\in\T_m}C(J).$ 
\item
The family $\T$ differentiates $L^1(X,\mu).$
\een 
Given an $\alpha$-tree $\T$ on $(X,\mu),$ a function $\varphi$ on $X$ is called $\T$-simple, if there exists an $N\ge0$ such that $\varphi$ is constant $\mu$-{\it a.e.} on each element of $\T_N.$

Observe that each $C(J)$ is necessarily finite. We will refer to the elements of $C(J)$ as children of $J$ and to $J$ as their parent. Also note that $\T(J):=\{I\in\T: I\subset J\}$ is an $\alpha$-tree on $(J,\mu|_J).$ We write $\mathcal{T}_k(J)$ for the collection of all descendants of $J$ of the $k$-th generation relative to $J;$ thus, $\T(J)=\bigcup_m\T_k(J).$

\end{definition}

\begin{remark}
The definition just given is similar to the one used in Bellman-function contexts by Melas \cite{melas} and Melas, Nikolidakis, and Stavropoulos \cite{mns}. In particular, like those authors, we restrict the size of children of each element of the tree. The main distinction is that in those applications the trees were assumed homogeneous, meaning that each element of the tree was split into the same number of children, and all children had the same measure. The prototypical such tree is the collection of all dyadic subcubes of a fixed cube in $\mathbb{R}^n;$ in our terminology it is a $2^{-n}$-tree. However, homogeneous trees are too rigid for our purposes. The elements of a general $\alpha$-tree have similar nesting properties, but this concept allows us to divide a parent into an arbitrary number of children of varying sizes, so long as none is too small. In addition, we do not foreclose the possibility that $\mu$ has atoms and so a parent can have only one child. 
\end{remark}

Suppose a measure space $(X,\mu)$ supports an $\alpha$-tree $\T$ for some $\alpha\in(0,1/2].$ There is a natural associated BMO:
$$
\varphi\in\BMO(\mathcal{T})\Longleftrightarrow \|\varphi\|_{\BMO(\mathcal{T})}:=\sup_{J\in\mathcal{T}}\{\av{\varphi^2}{J,\mu}-\av{\varphi}{J,\mu}^2\}^{1/2}<\infty.
$$
Let us make two formal definitions:
\eq[21]{
\ve_0^\alpha=\frac{\sqrt\alpha}{1-\alpha}\log(1/\alpha),\qquad K(\alpha,\ve)=\frac{1-\alpha}{e^{\sqrt\alpha\ve}-\alpha\, e^{\frac\ve{\sqrt\alpha}}}.
}
In this notation we have the following theorem.
\begin{theorem} 
\label{ta}
If $\alpha\in(0,1/2],$ $\ve\in(0,\ve_0^\alpha),$ $\T$ is an $\alpha$-tree on a measure space $(X,\mu),$ $J\in\mathcal{T},$ and $\varphi\in\BMO_\ve(\mathcal{T}),$ then
\eq[22]{
\av{e^{\varphi}}{J,\mu}\le K(\alpha,\ve)\,e^{\av{\varphi}{\scriptscriptstyle J,\mu}}.
}
\end{theorem}

Setting $\alpha=2^{-n}$ gives the values of $\ve_0^d(n)$ and $C(\ve,n)$ in Theorem~\ref{t1}. The fact that those values are sharp does not follow from Theorem~\ref{ta} and requires a separate construction of optimizers, i.e., functions from $\BMO^d$ for which equality is attained in~\eqref{i2}. This construction is carried out in Section~\ref{JN}.

Our technique works for other BMO inequalities as well. We illustrate this point by establishing an inequality relating $1$- and $2$-oscillations of BMO functions, which implies equivalence of the corresponding BMO norms. Specifically, for $\varphi\in\BMO^d(\rn),$ let
$$
\|\varphi\|_{\BMO^{d,1}(\rn)}=\sup_{J\in{\mathcal D}}\Delta_{1,J}(\varphi).
$$
Since $\Delta_{1,J}(\varphi)\le\big(\Delta_{2,J}(\varphi)\big)^{1/2},$ we have $\|\varphi\|_{\BMO^{d,1}(\rn)}\le\|\varphi\|_{\BMO^d(\rn)}.$ Importantly, both inequalities can be reversed.
\begin{theorem}
\label{t6}
If $\varphi\in\BMO^d(\rn),$ then for any $J\in\mathcal{D},$
\eq[i6]{
\frac{2^{n/2}}{2^n+1}\,\Delta_{2,J}(\varphi)\le\|\varphi\|_{\BMO^d(\rn)}\,\Delta_{1,J}(\varphi).
}
This inequality us sharp for every value of $\|\varphi\|_{\BMO^d(\rn)}.$ Consequently\textup,
\eq[i7]{
\frac{2^{n/2}}{2^n+1}\,\|\varphi\|_{\BMO^d(\rn)}\le\|\varphi\|_{\BMO^{d,1}(\rn)}.
}
\end{theorem}
As before, this theorem is a partial corollary of the corresponding result for $\alpha$-trees.
\begin{theorem}
\label{thb}
If $\T$ is an $\alpha$-tree on a measure space $(X,\mu),$ then for any $\ve>0,$ any 
$\varphi\in\BMO_\ve(\T),$ and any $J\in\mathcal{T},$
\eq[mest11]{
\frac{\sqrt\alpha}{(1+\alpha)\ve}\,\Delta_{2,\mu,J}(\varphi)\le \Delta_{1,\mu,J}(\varphi).
}
Consequently\textup,
\eq[mest22]{
\frac{\sqrt\alpha}{1+\alpha}\,\|\varphi\|_{\BMO(\T)}\le \|\varphi\|_{\BMO^1(\T)},
}
where $\|\varphi\|_{\BMO^1(\T)}=\sup_{J\in{\mathcal T}}\Delta_{1,\mu,J}(\varphi).$
\end{theorem}

Before proving all results just stated, let us put them in historical and methodological context. 
The notion of a Bellman function in analysis goes back to Bukholder~\cite{b} and Nazarov, Treil, and Volberg \cite{ntv1}, \cite{nt}, \cite{ntv2}. In the original, utilitarian meaning, {\it a} Bellman function was an inductive device with certain size and convexity properties that allowed one to estimate an integral functional by induction on dyadic or pseudo-dyadic scales in the underlying measure space. In a more recent understanding, {\it the} Bellman function for an inequality is the solution of the corresponding extremal problem (and, often, also is a solution of the homogeneous \ma equation on a Euclidean domain). It turns out that such a solution, once in hand, not only allows one to perform the induction on scales, but also encodes information about optimizing functions or sequences. 

The specific Bellman functions we use here have origins in our studies~\cite{sv} and~\cite{sv1} and, like those earlier functions, they are defined on a parabolic domain in the plane. However, those papers dealt with the one-dimensional BMO on an interval. This meant that on each step of the induction one would split an interval into two subintervals, yielding three points in the Bellman domain that had to be controlled using the concavity/convexity of the Bellman function. In dimension $n$ there are $2^n+1$ such points, and since the domain is non-convex, it is not clear how to control all of them at the same time. This has led to the introduction of $\alpha$-trees and the notion of 
$\alpha$-concave/$\alpha$-convex functions, defined the next section. If one has such a tree and a function, one can run the induction three points at a time. We do not provide a general recipe for constructing $\alpha$-concave/convex functions, but simply present natural $\alpha$-tree analogs of the Bellman functions from~\cite{sv} and~\cite{sv1}. With (perhaps considerable) effort, one can ``$\alpha$-ize'' any Bellman function for $\BMO(\mathbb{R}),$ and these are now plentiful: following~\cite{sv1}, the papers~\cite{iosvz1} and \cite{iosvz2} develop a rather general method for computing them; a somewhat different example is given in~\cite{vv}.

Though motivated by inequalities for the dyadic $\BMO(\mathbb{R}^n),$ Bellman analysis on $\alpha$-trees has wider applications. First, it works in other geometric settings; one important example is supplied by spaces of homogeneous type, where the ``dyadic cubes'' of M.~Christ~\cite{christ} give rise to $\alpha$-trees with $1/\alpha$ comparable to the doubling constant. Second, it naturally extends to other function classes that yield non-convex Bellman domains, such as  $A_p$ and reverse H\"older classes, $A_p$-weighted $L^p,$ etc. Lastly, when properly modified to include the range $\alpha\in(1/2,1]$ it may allow one to obtain results for continuous BMO simply by taking a suitable supremum in $\alpha.$ In this modification, described in Section~\ref{martingales}, each BMO function generates its own tree and thus yields what we call an $\alpha$-martingale.

The rest of the paper is organized as follows: in Section~\ref{tools} we define $\alpha$-concave/convex functions and formalize Bellman induction on $\alpha$-trees using such functions. We also present a set of easy-to-verify sufficient conditions for a function to be $\alpha$-concave/convex. In Section~\ref{JN} we prove Theorem~\ref{ta}, which allows us to compute the exact Bellman function for the John--Nirenberg inequality for $\BMO^d(\mathbb{R}^n);$ Theorem~\ref{t1} then follows easily. This sequence is repeated in Section~\ref{1-2}: we first prove Theorem~\ref{thb}, then define and find the corresponding Bellman function, which immediately gives Theorem~\ref{t6}. Finally, in Section~\ref{martingales} we introduce the notion of an $\alpha$-martingale, state a general result about such martingales generated by BMO functions, and estimate the John--Nirenberg constant for the continuous $\BMO(\mathbb{R}^n)$ in terms of a special martingale-related parameter of the space.

\section{Main technical tools}
\label{tools}
\noindent For $\ve>0,$ let
$$
\Oe=\{x=(x_1,x_2):~x_1^2\le x_2\le x_1^2+\ve^2\}.
$$
Also, for any $\xi\ge0,$ let $\Gamma_\xi=\{x:x_2=x_1^2+\xi^2\};$ thus, $\Gamma_0$ and $\Gamma_\ve$ are the lower and upper boundaries of $\Oe,$ respectively. 
A line segment connecting points $x$ and $y$ in the plane will be denoted by $[x,y]$ and the length of such a segment, by $\len{x}{y}.$ 

Observe that if $\T$ is an $\alpha$-tree on $(X,\mu)$ and $\varphi\in\BMO_\ve(\T),$ then for any $J\in\T,$ 
$\avm{\varphi}J^2\le\avm{\varphi^2}J\le\avm{\varphi}J^2+\ve^2$ and so the point $(\avm{\varphi}J,\avm{\varphi^2}J)$ is in $\Oe.$ This simple fact is the geometric foundation of the Bellman approach to the $L^2$-based $\BMO,$ which consists of using an appropriate concave or convex function on $\Oe$ to perform induction on the generation of the tree and bound the desired integral in the limit. However, since $\Oe$ is itself non-convex, we need to strengthen the usual notion of concavity/convexity.

\begin{definition}
A function on $\Oe$ is called locally concave (respectively, locally convex), if it is concave (respectively, convex) on any convex subset of $\Oe.$
\end{definition}

\begin{definition}
\label{def}
If $\alpha\in\big(0,\frac12\big],$ a function $B$ on $\Oe$ is called $\alpha$-concave if
\eq[601]{
B(\beta x^-+(1-\beta) x^+)\ge \beta B(x^-)+(1-\beta) B(x^+),\\
}
for any $\beta\in\big[\alpha,\frac12\big]$ and any two points $x^\pm\in\Oe$ such that $\beta x^-+(1-\beta) x^+\in\Oe.$

Similarly, $b$ is called $\alpha$-convex on $\Oe$ if 
\eq[602]{
b(\beta x^-+(1-\beta) x^+)\le \beta b(x^-)+(1-\beta) b(x^+),\\
}
for all $\beta, x^-,$ and $x^+$ as above.
\end{definition}

Armed with such a function, we can obtain the desired integral estimate using what is commonly referred to as Bellman induction. The procedure depends on a simple geometric fact that replaces the average of $N$ points in $\Oe$ with the average of just two points.
\begin{lemma}
\label{L0}
If $N\ge2,$ numbers $\alpha_1,...,\alpha_N\in(0,1)$ are such that $\sum_{k=1}^N\alpha_k=1,$  and points $P_1,...,P_N\in\Oe$ are such that $\sum_{k=1}^N\alpha_kP_k\in\Oe,$  then there exists at least one $j,$ $1\le j\le N,$ for which the point
$
R_j:=\frac1{1-\alpha_j}\,\sum_{k\ne j} \alpha_kP_k
$
is in~$\Oe.$
\end{lemma}
\begin{proof}
Assume, to the contrary, that $R_j\notin\Omega_\ve,\forall j=1,...,N.$ Since the set $\{x:x_2\ge x_1^2\}$ is convex, all $R_j$ are in this set. Since none is in $\Oe,$ all are, in fact, in the set $\{x:x_2>x_1^2+\ve^2\},$ which is also convex. Therefore, their convex combination,
$$
\frac1{N-1}\sum_{j=1}^N(1-\alpha_j) R_j=\frac1{N-1}\sum_{j=1}^N\Big(\sum_{k=1}^N\alpha_kP_k-\alpha_jP_j\Big)=\sum_{k=1}^N\alpha_kP_k,
$$
is also in $\{x:x_2>x_1^2+\ve^2\},$ which is a contradiction.
\end{proof}
\begin{lemma}
\label{L1}
Take $\alpha\in(0,1/2]$ and let $\T$ be an $\alpha$-tree on a measure space $(X,\mu).$ Let $\varphi$ be a $\T$-simple function\textup; set  $\ve=\|\varphi\|_{\BMO(\T)}$ \textup(note that $\varphi$ is bounded and thus in $\BMO(\T)$\textup).
\ben
\item[(i)]
If $B$ is an $\alpha$-concave function on $\Oe,$ then
$$
B\big(\avm{\varphi}X,\avm{\varphi^2}X\big)\ge \frac1{\mu(X)}\,\int_X B(\varphi,\varphi^2)\,d\mu.
$$
\item[(ii)]
If $b$ is an $\alpha$-convex function on $\Oe,$ then
$$
b\big(\avm{\varphi}X,\avm{\varphi^2}X\big)\le \frac1{\mu(X)}\,\int_X b(\varphi,\varphi^2)\,d\mu.
$$
\een
\end{lemma}
\begin{proof}

We will prove only statement~(i); the proof of~(ii) is the same, except all inequality signs are reversed. For all $I\in\T,$ let $P_I=\big(\avm{\varphi}I,\avm{\varphi^2}I\big)$ and $\mu_I=\mu(I).$ Since $\varphi\in\BMO_\ve(\T),$ we have $P_I\in\Oe.$ Furthermore, 
$
P_I=\frac1{\mu_I}\sum_{J\in\T_1(I)}\mu_J\,P_J.
$

We first claim that for any $I\in\T,$
\eq[ind]{
B(P_I)\ge\frac1{\mu_I}\sum_{J\in\T_1(I)}\mu_J\,B(P_J).
}
If $\T_1(I)$ has only one element, then~\eqref{ind} holds with equality. Assume that $\T_1(I)$ has two or more elements. By Lemma~\ref{L0}, there exists an $L\in\T_1(I)$ such that the point
$$
R_{L}:=\frac1{\mu_I-\mu_L}\,\sum_{\substack{ J\in\T_1(I) \\ J\ne L}}\mu_J\,P_{J}
$$
is in $\Oe.$ We have 
$
P_I=\frac{\mu_L}{\mu_I}\,P_L+\big(1-\frac{\mu_L}{\mu_I}\big)\,R_L,
$
and, since $\T$ is an $\alpha$-tree, both $\mu_L/\mu_I\ge\alpha$ and $1-\mu_L/\mu_I\ge\alpha.$ By the $\alpha$-concavity of $B,$
$$
B(P_I)\ge \frac{\mu_L}{\mu_I}\,B(P_L)+\Big(1-\frac{\mu_L}{\mu_I}\Big)\,B(R_L).
$$
If $\T_1(I)$ has two elements, this is exactly statement~\eqref{ind}. If $\T_1(I)$ has more than two elements, we apply Lemma~\ref{L0} to $R_L$ in place of $P_I:$
$$
B(R_L)\ge \frac{\mu_K}{\mu_I-\mu_L}\,B(P_K)+\Big(1-\frac{\mu_K}{\mu_I-\mu_L}\Big)\,B(R_{L,K}),
$$
for some $K\in\T_1(I), K\ne L,$ and $R_{L,K}=\frac1{\mu_I-\mu_L-\mu_K}\,\sum_{J\ne L,K}\mu_J\,P_{J}.$ This gives
$$
B(P_I)\ge \frac{\mu_L}{\mu_I}\,B(P_L)+\frac{\mu_K}{\mu_I}\,B(P_K)+\Big(1-\frac{\mu_L+\mu_K}{\mu_I}\Big)\,B(R_{L,K}).
$$
Continuing in this fashion, we obtain~\eqref{ind}.

Now, let $N$ be such that $\varphi$ is constant on each $J\in\T_N.$ Note that this means that $\avm{\varphi^2}J=\avm{\varphi}J^2$ for all such $J.$ A repeated application of~\eqref{ind} gives
\begin{align*}
B(P_X)&\ge \frac1{\mu_X}\sum_{J\in\T_1}\mu_J\,B(P_J)\ge \frac1{\mu_X}\sum_{J\in\T_1}\,\sum_{R\in\T_1(J)}\mu_R\,B(P_R)
= \frac1{\mu_X}\sum_{J\in\T_2}\mu_J\,B(P_J)\\
&\ge\dots\ge  \frac1{\mu_X}\sum_{J\in\T_N}\mu_J\,B(P_J)= \frac1{\mu_X}\sum_{J\in\T_N}\mu_J\,B\big(\avm{\varphi}J,\avm{\varphi}J^2\big).
\end{align*}
The last expression is precisely $\frac1{\mu_X}\int_XB(\varphi,\varphi^2)\,d\mu,$ and the proof is complete.
\end{proof}
 
Our next lemma gives sufficient conditions for a function on $\Oe$ to be $\alpha$-concave/convex.
\begin{lemma}
\label{L2}
Let $\ve>0$ and $\alpha\in(0,\frac12].$
Assume that functions $B$ and $b$ on $\Oe$ satisfy the following three conditions\textup:
\ben
\item
\label{C1}
$B$ is locally concave on $\Oe$ and $b$ is locally convex on $\Oe,$
\item
\label{C2}
$B$ and $b$ have non-tangential derivatives at every point $\Gamma_\ve.$ Furthermore, for any two distinct points on $\Gamma_\ve,$  $P=(p,p^2+\ve^2)$ and $Q=(q,q^2+\ve^2)$ with $|p-q|\le\frac{1-\alpha}{\sqrt\alpha}\,\ve,$
\begin{align*}
(D_{\scriptscriptstyle\overrightarrow{PQ}}B)(P)&\ge (D_{\scriptscriptstyle\overrightarrow{PQ}}B)(Q),\\
(D_{\scriptscriptstyle\overrightarrow{PQ}}b)(P)&\le (D_{\scriptscriptstyle\overrightarrow{PQ}}b)(Q),
\end{align*}
where $D_{_{\overrightarrow{PQ}}}$ denotes the derivative in the direction of the vector $\overrightarrow{PQ}.$
\item
\label{C3}
For any $P$ and $Q$ as above, and $S=\frac1{1-\alpha}(P-\alpha Q),$
\begin{align*}
B(P)&\ge (1-\alpha)\, B(S)+\alpha\, B(Q),\\
b(P)&\le (1-\alpha)\, b(S)+\alpha\, b(Q).
\end{align*}
\een
Then $B$ is $\alpha$-concave on $\Oe$ and  $b$ is $\alpha$-convex on $\Oe.$
\end{lemma}
\begin{remark}
Before proving this lemma, let us examine its conditions more closely. First, Conditions~\eqref{C1} and~\eqref{C3} are clearly necessary for $\alpha$-concavity/convexity. Second, Condition~\eqref{C2} is also quite natural: along any line segment fully contained in $\Oe$ the directional derivatives of $B$ and $b$ (defined almost everywhere) are monotone. This condition can be seen as preserving this monotonicity of derivatives even for those segments that cross the boundary curve $\Gamma_\ve$ at two points, as long as the part external to $\Oe$ ($[P,Q]$in this case) is not too large. Note that $B$ and $b$ need not be defined along the segment $[P,Q]$ itself.

\end{remark}

\begin{proof}
Again, we only prove the statement for $B.$ If the segment $[x^-,x^+]$ is contained in $\Oe,$ inequality~\eqref{601} holds because $B$ is locally concave in $\Oe.$ Thus, we need to consider only those segments that cross $\Gamma_\ve$ at two points. 

Take any $\beta\in[\alpha,\frac12]$ and any $x^-,x^+\in\Oe$ such that $x^0:=\beta\, x^-+(1-\beta)\,x^+\in\Oe.$ Let $\ell$ be the line containing the three points $x^-,x^0,$ and $x^+$ and assume that $\ell$ intersects $\Gamma_\ve$ at two points,  $P$ and $Q.$ Then $\ell$ also intersects $\Gamma_0$ at two points, say $U$ and $W,$ named so that the order of points on $\ell$ is $U, P, Q,W.$ Without loss of generality, assume that $x^-,x^0\in[U,P]$ and $x^+\in[Q,W].$ Finally let $S\in[U,W]$ be such that $\len S P=\alpha\len S Q,$ that is $P=(1-\alpha)\,S+\alpha\,Q.$ Figure~\ref{fig1} shows the location of all these points for $\alpha=\frac16$ and $\beta=\frac14.$ 

\begin{figure}[ht]
  \centering{\includegraphics[width=12cm]{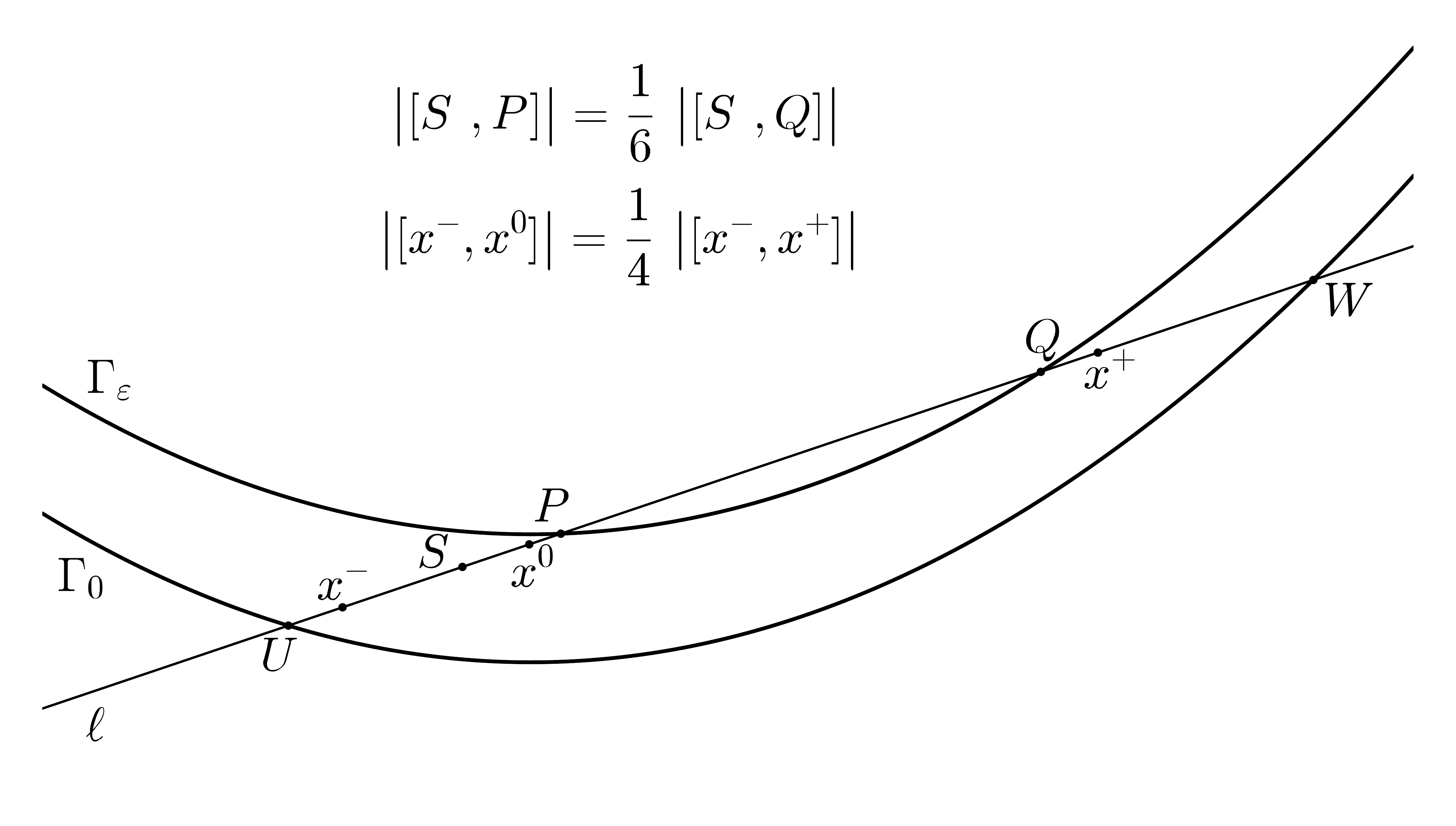}}
\caption{Picture for the proof of Lemma~\ref{L2} with $\alpha=\frac16$ and $\beta=\frac14.$}
  \label{fig1}
\end{figure}

The essence of the lemma is that if Conditions~\eqref{C1} and~\eqref{C2} are fulfilled, then inequality~\eqref{601} reduces to its special case assumed in Condition~\eqref{C3}, i.e., to the situation when $\beta=\alpha$ and $x^0,x^+\in\Gamma_\ve.$ In fact, it will be technically convenient to show a little more. Specifically, we claim that if
\eq[661]{
B(P)\ge (1-\alpha) B(S)+\alpha\,B(Q),
}
then 
\eq[662]{
B(x^0)\ge \frac{\len{x^0}{x^+}}{\len{x^-}{x^+}}B(x^-)+\frac{\len{x^-}{x^0}}{\len{x^-}{x^+}}B(x^+),
}
for any $x^-\in[U,S],$ $x^0\in[x^-,P],$ and $x^+\in[Q,W].$

Parametrize the segment $[U,W]$ by $x(t)=(1-t)U+tW,$ $0\le t\le 1.$ Let $t_-,t_S,t_0,t_P,t_Q,$ and $t_+$ be the values of the parameter corresponding to the points $x^-,S,x^0,P,Q,$ and $x^+,$ respectively. Let $A(t)=B(x(t))$ and define a new function $F$ by
$$
F(t_-,t_0,t_+)=(t_+-t_-)A(t_0)-(t_+-t_0)A(t_-)-(t_0-t_-)A(t_+)
$$
on the domain
$$
D=\{(t_-,t_0,t_+):0\le t_-\le t_S,\quad t_-\le t_0\le t_P,\quad t_Q\le t_+\le 1\}.
$$
In this notation, condition~\eqref{661} is equivalent to
\eq[663]{
F( t_S,t_P,t_Q)\ge0,
}
and to prove that~\eqref{661} implies~\eqref{662} is the same as to prove that~\eqref{663} implies
\eq[664]{
F(t_-,t_0,t_+)\ge0,\quad \forall (t_-,t_0,t_+)\in D.
}
Thus, assume that~\eqref{663} holds.

Since $B$ is concave in $\Oe,$ $A$ is concave on $[0,t_P]$, and, viewed as a function of $t_0$, $F$ is just an affine transformation of $A$, with a positive coefficient. Therefore, $F$ is concave in $t_0$ on $[t_-,t_P],$ and $\min_{D} F=\min\{F(t_-,t_-,t_+),F(t_-,t_P,t_+)\}.$ Since $F(t_-,t_-,t_+)=0,$ to prove~\eqref{663} we need to show that $F(t_-,t_P,t_+)\ge0.$

We have
\begin{align*}
F(t_-,t_P,t_+)&=(t_+-t_-)A(t_P)-(t_+-t_P)A(t_-)-(t_P-t_-)A(t_+)\\
&=(t_+-t_P)(t_P-t_-)\,\left[\frac{A(t_P)-A(t_-)}{t_P-t_-}\right]-(t_P-t_-)(A(t_+)-A(t_P))\\
&\ge(t_+-t_P)(t_P-t_-)\,\left[\frac{A(t_P)-A(t_S)}{t_P-t_S}\right]-(t_P-t_-)(A(t_+)-A(t_P))\\
&=\frac{t_P-t_-}{t_P-t_S}\,F(t_S,t_P,t_+).
\end{align*}
where we used the concavity of $A$ on $[0,t_P].$ To estimate~$F(t_S,t_P,t_+),$ we write:
$$
F(t_S,t_P,t_+)=F(t_S,t_P,t_Q)+(t_+-t_Q)(t_P-t_S)\left[\frac{A(t_P)-A(t_S)}{t_P-t_S}-\frac{A(t_+)-A(t_Q)}{t_+-t_Q}\right].
$$
The first term is non-negative by~\eqref{663}, while the expression in brackets can be estimated by $A'(t_P)-A'(t_Q),$ which is non-negative by Condition~\eqref{C2}. This completes 
the proof.
\end{proof}

\section{The John--Nirenberg inequality}
\label{JN}

Here we prove Theorems~\ref{t1} and~\ref{ta} (in the reverse order), by choosing the smallest $\alpha$-concave element from a family of locally concave functions developed in~\cite{sv}. We also obtain another significant result, not stated in the introduction: Theorem~\ref{thB} gives the exact Bellman function for the John--Nirenberg inequality on $n$-dimensional dyadic BMO.
 
Recall definitions~\eqref{21} of $\ve_0^\alpha$ and $K(\alpha,\ve).$ Let 
$$
g(\alpha,\delta,\ve)=\frac{1-\rtde}{1-\delta}e^{-\delta+\rtde}-K(\alpha,\ve).
$$
\begin{lemma}
\label{LVA}
For each $\alpha\in(0,\frac12]$ and each $\ve\in(0,\ve_0^\alpha)$ the equation $g(\alpha,\delta,\ve)=0$ has a unique solution $\delta(\alpha,\ve)$ such that 
\eq[de]{
\ve<\delta(\alpha,\ve)<\min\Big\{1,\frac{1+\alpha}{2\sqrt\alpha}\,\ve\Big\}.
}
\end{lemma}
\begin{proof}
For any $\delta\in(0,1),$ we have 
$$
\frac{\partial g}{\partial\delta}=\frac{\delta(\delta-\rtde)}{(1-\delta)^2}\,e^{-\delta+\rtde}>0.
$$
We need to check the sign of $g(\alpha,\delta,\ve)$ at the two endpoints, $\delta=\ve$ and $\delta=\min\big\{1,\frac{1+\alpha}{2\sqrt\alpha}\,\ve\big\}.$ For $\delta=\ve$ we compute:
\begin{align*}
g(\alpha,\ve,\ve)&=\frac{e^{-\sqrt\alpha\ve}}{(1-\ve)\big(1-\alpha e^{\frac{1-\alpha}{\sqrt\alpha}\ve}\big)}\,
\left[e^{-(1-\sqrt\alpha)\ve}-\alpha e^{(\frac1{\sqrt\alpha}-1)\ve}-(1-\alpha)(1-\ve)\right]\\
&=\frac{e^{-\sqrt\alpha\ve}}{(1-\ve)\big(1-\alpha e^{\frac{1-\alpha}{\sqrt\alpha}\ve}\big)}\,
\left[\sum_{k=3}^\infty\frac{\ve^k}{k!}\,(1-\sqrt\alpha)^k\big((-1)^k-\alpha^{1-k/2}\big)\right]<0.
\end{align*}
To check the right endpoint, assume first that $0<\ve<\frac{2\sqrt\alpha}{1+\alpha}.$ Then, after a bit of algebra, we obtain
\begin{align*}
g\Big(\alpha,\frac{1+\alpha}{2\sqrt\alpha}\,\ve,\ve\Big)&=
\frac{2\sqrt\alpha-(1-\alpha)\ve}{2\sqrt\alpha-(1+\alpha)\ve}\,e^{-\sqrt\alpha\ve}-K(\alpha,\ve)\\
&=\frac{2\alpha\,e^{\frac{1-3\alpha}{2\sqrt\alpha}\ve}}{\big(1-\frac{1+\alpha}{2\sqrt\alpha}\ve\big)\big(1-\alpha e^{\frac{1-\alpha}{\sqrt\alpha}\ve}\big)}\,
\left[\Big(\frac{1-\alpha}{2\sqrt\alpha}\,\ve\Big)\cosh\Big(\frac{1-\alpha}{2\sqrt\alpha}\,\ve\Big)-\sinh\Big(\frac{1-\alpha}{2\sqrt\alpha}\,\ve\Big)\right]>0.
\end{align*}
Thus, the interval $\big(\ve,\frac{1+\alpha}{2\sqrt\alpha}\,\ve\big)$ contains a unique root $\delta$ of the equation $g(\alpha,\delta,\ve)=0.$

Now, assume that $\frac{2\sqrt\alpha}{1+\alpha}\le\ve<\ve_0(\alpha).$ Since $g(\alpha,\delta,\ve)\to\infty$ as $\delta\to1,$ we conclude that there exists a unique root $\delta\in(\ve,1).$
\end{proof}

Let
\eq[ba0]{
B_\delta(x)=\frac{e^{-\delta}}{1-\delta}\,e^{x_1+r(x)}(1-r(x)),\qquad \text{where}~r(x)=\sqrt{\delta^2-x_2+x_1^2}.
}
This family of functions was obtained in~\cite{sv}. As shown there, for every~$\delta\ge0$ $B_\delta$ is a solution of the \ma equation $B_{x_1x_1}B_{x_2x_2}=B_{x_1x_2}^2$ on $\Oe$ satisfying $B_\delta(x_1,x_1^2)=e^{x_1}.$ Furthermore, $B_\delta$ is locally concave on $\Oe.$
In fact, more can be said if $\delta$ is sufficiently large.

For $\ve>0$ and $0<\alpha\le1/2,$ let $\delta(\alpha,\ve)$ be given by Lemma~\ref{LVA}. Define
\eq[ba]{
B_{\alpha,\ve}(x)=B_{\delta(\alpha,\ve)}(x),~x\in\Oe.
}
Since $\ve<\delta(\alpha,\ve)<1,$ we have $B_{\alpha,\ve}\in C^\infty(\Oe).$ 
\begin{lemma}
The function $B_{\alpha,\ve}$ is $\alpha$-concave on $\Oe.$ 
\label{lba}
\end{lemma}
\begin{proof}
Let us write simply $\delta$ for $\delta(\alpha,\ve)$ and $B$ for $B_{\alpha,\ve}.$ We have to verify the three conditions of Lemma~\ref{L2} for 
$B.$ Let $\mu=\sqrt{\delta^2-\ve^2}.$ Observe that $\mu<1$ and that $r(x)=\mu$ for $x\in\Gamma_\ve.$ 

Condition~\eqref{C1}: Direct differentiation (or Lemma~3c of~\cite{sv}) shows that $B$ is locally concave in $\Oe.$

To verify Condition~\eqref{C2}, take any $p$ and $q$ such that $0<|p-q|\le\frac{1-\alpha}{\sqrt\alpha}\,\ve$ and let $P=(p,p^2+\ve^2),$ $Q=(q^2,q^2+\ve^2),$ and $\xi=(q-p)/2.$ We have
$$
\nabla B(x)=\frac{e^{-\delta}}{1-\delta}\,e^{x_1+r(x)}
\begin{bmatrix}
1-r(x)-x_1\\
\frac12
\end{bmatrix},
\qquad
\overrightarrow{PQ}=(q-p)
\begin{bmatrix}
1\\
p+q
\end{bmatrix}.
$$
Therefore, for some non-negative multiple $C,$
\begin{align*}
(D_{\scriptscriptstyle\overrightarrow{PQ}}B)(P)-(D_{\scriptscriptstyle\overrightarrow{PQ}}B)(Q)&=
C(q-p)\left[e^p\Big(1-\mu+\frac{q-p}2\Big)-e^q\Big(1-\mu-\frac{q-p}2\Big)\right]\\
&=4Ce^{(p+q)/2}\left[\xi^2\cosh\xi-(1-\mu)\xi\sinh\xi\right]\\
&\ge4Ce^{(p+q)/2}\left[\xi^2\cosh\xi-\xi\sinh\xi\right]\ge0.
\end{align*}

Lastly, to verify Conidtion~\eqref{C3}, take $P$ and $Q$ as above; let $S=\frac1{1-\alpha}(P-\alpha Q)$ and $\theta=\frac{q-p}{1-\alpha}.$ We have 
$S=\frac1{1-\alpha}(p-\alpha q,p^2-\alpha q^2+(1-\alpha)\ve^2)$, so 
$r(S)=\sqrt{\delta^2-s_2+s_1^2}=\sqrt{\mu^2+\alpha\theta^2}.$ Hence,
\begin{align*}
B(P)&-(1-\alpha)\,B(S)-\alpha\,B(Q)\\
&=
\frac{e^{-\delta}}{1-\delta}\,\left[ e^{p+\mu}(1-\mu)-(1-\alpha)\,e^{\frac{p-\alpha q}{1-\alpha}+r(S)}(1-r(S))-\alpha\,e^{q+\mu}(1-\mu)\right]\\
&=\frac{e^{-\delta}}{1-\delta}\,e^q\,\left[ e^{\mu}(1-\mu)(e^{-(1-\alpha)\theta}-\alpha)-
(1-\alpha)\,e^{-\theta+\sqrt{\mu^2+\alpha\theta^2}}
\left(1-\sqrt{\mu^2+\alpha\theta^2}\right)\right]\\
&=: \frac{e^{-\delta}}{1-\delta}\,e^q\,G(\theta).
\end{align*}
We now need to show that the function $G$ is non-negative on the domain $\big[-\frac{\ve}{\sqrt\alpha},\frac{\ve}{\sqrt\alpha}\big].$ For $t\ge0,$ let $h(t)=e^t(1-t).$ An easy computation gives
$$
G'(\theta)=e^{-(1-\alpha)\theta}\left[h\left(-\alpha\theta+\sqrt{\mu^2+\alpha\theta^2}\right)-h(\mu)\right].
$$
Since $h$ is a decreasing function, if $\theta\le0,$ then $G'(\theta)\le0.$ Since $G(0)=0,$ we conclude that $G(\theta)\ge0$ on $[-\frac{\ve}{\sqrt\alpha},0].$ 

The interval $[0,\frac{\ve}{\sqrt\alpha}]$ requires a little more care. To determine the sign of $G'$ we need to compare the quantities $\mu$ and $-\alpha\theta+\sqrt{\mu^2+\alpha\theta^2}$ (note that the latter is non-negative):
$
-\alpha\theta+\sqrt{\mu^2+\alpha\theta^2}\le \mu \Longleftrightarrow \theta\le \frac{2\mu}{1-\alpha}.
$ 
In addition, $\frac{2\mu}{1-\alpha}=\frac{2\sqrt{\delta^2-\ve^2}}{1-\alpha}\le\frac\ve{\sqrt\alpha}$ (this inequality is equivalent to the right-hand inequality in~\eqref{de}). Therefore, $G'\ge0$ on  $[0,\frac{2\mu}{1-\alpha}]$ and $G'\le0$ on $[\frac{2\mu}{1-\alpha},\frac\ve{\sqrt\alpha}].$ Since $G(0)=0,$ it remains to check that $G(\frac\ve{\sqrt\alpha})\ge0.$ In fact, we have $G(\frac\ve{\sqrt\alpha})=0,$ since this is the (slightly rewritten) equation~$g(\alpha,\delta,\ve)=0.$ This completes the proof.
\end{proof}

Theorem~\ref{ta} now follows easily.
\begin{proof}[Proof of Theorem~\ref{ta}]
Take any $\varphi\in\BMO_\ve(\T)$ and for $N>0$ let $\varphi_N$ be the truncation of $\varphi$ at the $N$-th generation of $\T,$ $\varphi_N=\sum_{J\in\T_N}\avm{\varphi}J\chi^{}_J.$ Note that $\avm{\varphi^{}_N}X=\avm{\varphi}X$ and $\avm{\varphi^2_N}X\le\avm{\varphi^2}X.$

Since $B_{\alpha,\ve}$ is $\alpha$-concave on $\Oe$ and $B_{\alpha,\ve}(x_1,x_1^2)=e^{x_1},$ we can apply Lemma~\ref{L1}:
$$
B_{\alpha,\ve}\big(\avm{\varphi^{}_N}X,\avm{\varphi_N^2}X\big)\ge \frac1{\mu(X)}\,\int_X B_{\alpha,\ve}(\varphi^{}_N,\varphi_N^2)\,d\mu=\frac1{\mu(X)}\,\int_X e^{\varphi^{}_N}\,d\mu.
$$
Since the family $\T$ differentiates $L^1(X,\mu),$ $\varphi_N\to\varphi$ $\mu$-{\it a.e.} as $N\to\infty.$ Thus, $\avm{\varphi^2_N}X\to\avm{\varphi^2}X,$ by Fatou's Lemma. Because $B$ is continuous on $\Oe,$ the left-hand side tends to $B_{\alpha,\ve}(\avm{\varphi}X,\avm{\varphi^2}X).$ Hence, again by Fatou's Lemma, $e^\varphi$ is integrable and
\eq[i1]{
B_{\alpha,\ve}\big(\avm{\varphi}X,\avm{\varphi^2}X\big)\ge \avm{e^\varphi}X.
}
For a fixed $x_1,$ $B$ attains its maximum on $\Gamma_\ve,$ thus the left-hand side is bounded by
$$
e^{\av{\varphi}{\scriptscriptstyle X,\mu}}\,\frac{1-\rtde}{1-\delta}\,e^{-\delta+\rtde}=e^{\av{\varphi}{\scriptscriptstyle X,\mu}}\,K(\alpha,\ve).
$$
This proves~\eqref{22} with $J=X.$ The full result follows, since the set $\mathcal{T}(J)=\{I\in\mathcal{T}: I\subset J\}$ is itself an $\alpha$-tree and $\|\varphi\|_{\BMO(\T(J))}\le \|\varphi\|_{\BMO(\T)}.$ 
\end{proof}

Observe that~\eqref{i1} gives more than the John--Nirenberg inequality~\eqref{22}: with 
$\|\varphi\|_{\BMO(\T)}$ fixed, we get a continuum of more precise inequalities indexed by $(\avm{\varphi}X,\avm{\varphi^2}X).$ In the dyadic case, all these inequalities turn out to be sharp. We can state this compactly using the concept of a Bellman function.

Take $\ve>0$ and a dyadic cube $Q\subset\rn.$ For every $x\in\Oe,$ let
$$
\bel{B}^n_{\ve}(x)=\sup_{\varphi\in \BMO^d_{\ve}(Q)}\left\{\av{e^\varphi}Q:~
\av{\varphi}Q=x_1, \av{\varphi^2}Q=x_2\right\}.
$$
The function $\bel{B}^n_\ve$ is called the Bellman function for the integral John--Nirenberg inequality for $\BMO^d(\rn).$ It is easy to see that the supremum above is taken over a non-empty set for every $x\in\Oe.$ By rescaling, $\bel{B}^n_\ve$ can be seen not to depend on $Q.$

One of the main results of~\cite{sv} was the following theorem.
\begin{theorem}[\cite{sv}]
If $\ve<\sqrt2\log2,$ then
$$
\bel{B}^1_\ve(x)=B_{\frac12,\ve}(x),\quad\forall x\in\Oe.
$$
If $\ve\ge\sqrt2\log2,$ then
$$
\bel{B}_\ve^1(x)=
\begin{cases}
e^{x_1},& x\in\Gamma_0,\\
\infty,&x\in\Oe\setminus\Gamma_0.
\end{cases}
$$
\end{theorem}
We now have the following theorem.
\begin{theorem}
\label{thB}
Take $n\ge1.$ If $\ve<\ve_0^d(n),$ then
$$
\bel{B}^n_\ve(x)=B_{2^{-n},\ve}(x),\quad\forall x\in\Oe.
$$
If $\ve\ge\ve_0^d(n),$ then
$$
\bel{B}_\ve^n(x)=
\begin{cases}
e^{x_1},& x\in\Gamma_0,\\
\infty,&x\in\Oe\setminus\Gamma_0.
\end{cases}
$$
\end{theorem}

A key role in the proof of this theorem is played by the following function from $\BMO^d((0,1)^n).$ Let $Q=(0,1)^n$ and $Q_k=(0,2^{-k})^n$ for all $k\ge0.$ Now, let
\eq[i3]{
\varphi_*(t)=2^{-n/2}\,\big(k(2^n-1)-1\big),\quad t\in Q_k\setminus Q_{k+1}.
}
\begin{lemma}
We have $\av{\varphi_*}Q=0,$ $\av{\varphi_*^2}Q=1,$ and $\varphi_*\in\BMO^d(Q)$ with 
$\|\varphi_*\|_{\BMO^d(Q)}=1.$
\end{lemma}
\begin{proof}
It is clear that to verify that $\varphi_*\in\BMO^d(Q),$ and to compute the norm, it suffices to check the average oscillations of $\varphi_*$ only over the cubes $Q_k.$ We have
\begin{align*}
\av{\varphi_*}{Q_k}&=2^{n(k-1/2)}\sum_{j=k}^\infty \big(2^{-nj}-2^{-n(j+1)}\big)\big(j(2^n-1)-1\big)\\
&=2^{-n/2}(2^n-1)k
\end{align*}
and
\begin{align*}
\av{\varphi_*^2}{Q_k}&=2^{n(k-2)}(2^n-1)\sum_{j=k}^\infty 2^{-nj}\big(j(2^n-1)-1\big)^2\\
&=1+2^{-n}(2^n-1)^2k^2.
\end{align*}
Setting $k=0,$ we obtain $\av{\varphi_*}Q=0$ and $\av{\varphi_*^2}Q=1.$ Furthermore, for any $k,$ 
$\av{\varphi_*^2}{Q_k}-\av{\varphi_*}{Q_k}^2=1,$ which means that $\varphi_*\in\BMO^d(Q)$ and $\|\varphi_*\|_{\BMO^d(Q)}=1,$ as claimed.
\end{proof}

\begin{proof}[Proof of Theorem~\ref{thB}]
Take any $\ve\in(0,\ve_0^d(n))$ and let us write $\bel{B}$ for $\bel{B}^n_\ve,$ $B$ for $B_{2^{-n},\ve},$ and $\delta$ for $\delta(2^{-n},\ve).$ Fix a point $x\in\Oe$ and pick any $\varphi\in\BMO^d(Q)$ such that $(\av{\varphi}Q,\av{\varphi^2}Q)=x.$ By~\eqref{i1}, 
$$
B(x)\ge\av{e^\varphi}Q.
$$
Taking the supremum over all such $\varphi,$ we conclude that $B(x)\ge\bel{B}(x).$

To show the converse, we make use of the function $\varphi_*$ given by~\eqref{i3}. Take any $a\in\mathbb{R}$ and consider the corresponding point on $\Gamma_\ve,$ $P_a=(a,a^2+\ve^2).$ Let $\varphi_a=\ve\varphi_*+a.$ Then $(\av{\varphi_a}Q,\av{\varphi_a^2}Q)=P_a$ and $\|\varphi_a\|_{\BMO^d(Q)}=\ve.$ Furthermore,
\begin{align*}
\av{e^{\varphi_a}}{Q}&=e^{a-\ve 2^{n/2}}(1-2^{-n})\sum_{j=0}^\infty 2^{-nj} e^{\ve 2^{-n/2}(2^n-1)j}\\
&=e^{-\ve 2^{n/2}}(1-2^{-n})\sum_{j=0}^\infty \Big(2^{-n} e^{\ve 2^{-n/2}(2^n-1)}\Big)^j.
\end{align*}
This sum converges if and only if $\ve<\frac{2^{n/2}}{2^n-1}\,n\log 2=\ve_0^d(n),$ in which case
$$
\av{e^{\varphi_a}}{Q}=e^{a-\ve 2^{n/2}}(1-2^{-n})\frac1{1-2^{-n} e^{\ve 2^{-n/2}(2^n-1)}}=e^a\,K(2^{-n},\ve)=B(P_a).
$$
Therefore, $\bel{B}(P_a)\ge B(P_a).$

Thus, $\bel{B}=B$ on $\Gamma_\ve.$ Similarly, by considering the constant function $\psi_u=u$ corresponding to a point $(u,u^2)\in\Gamma_0,$ we conclude that $\bel{B}=B$ on $\Gamma_0.$ Now, as shown in~\cite{sv}, for any number $c,$ the function $B$ is linear along the line $\ell_c$ with the equation $x_2=2c x_1+\delta^2-c^2;$ such lines are tangent to $\Gamma_\delta$ and they foliate $\Oe.$ On the other hand, it is easy to show that $\bel{B}$ is locally concave in $\Oe$ (this is done is~\cite{sv} for $n=1$). 

Take any point $x\in\Oe$ and let $\ell$ be the unique tangent to $\Gamma_\delta$ passing through $x.$ Then $\ell$ intersects $\Gamma_0$ and $\Gamma_\ve$ at some points $U$ and $V,$ respectively, such that the segment $[U,V]$ lies entirely in $\Oe.$ We can write $x$ as a convex combination of $U$ and $V:$ 
$x=(1-\gamma)U+\gamma V$ for some $\gamma\in[0,1].$ Then,
$$
\bel{B}(x)\ge (1-\gamma)\,\bel{B}(U)+\gamma\,\bel{B}(V)=(1-\gamma)\,B(U)+\gamma\,B(V)=B(x).
$$
Thus $\bel{B}\ge B$ everywhere in $\Oe.$

Now take $\ve\ge\ve_0^d(n).$ If $x\in\Gamma_0,$ then $\bel{B}(x)=e^{x_1}.$ Indeed, if $\av{\varphi^2}Q=\av{\varphi}Q^2,$ then $\varphi$ is {\it a.e.} constant on $Q,$ and so $\av{e^\varphi}Q=e^{\av{\varphi}{\scriptscriptstyle Q}}.$ On the other hand, $\av{e^{\varphi_a}}{Q}=\infty,$ which means that $\bel{B}=\infty$ on $\Gamma_\ve.$ It is now an easy exercise to show that $\bel{B}=\infty$ everywhere in~$\Oe\setminus\Gamma_0.$
\end{proof}

Theorem~\ref{t1} now follows.

\begin{proof}[Proof of Theorem~\ref{t1}]
Note that this theorem is stated for $\BMO^d(\rn),$ while both Theorem~\ref{ta} and Theorem~\ref{thB} deal with $\BMO$ on sets of finite measure. However, the difference is inconsequential. Indeed, take any $\varphi\in\BMO^d(\rn)$ and let $\ve=\|\varphi\|_{\BMO^d(\rn)}.$ Assume that $\ve<\ve_0^d(n).$ Take any $Q\in\mathcal{D}.$ Then $\|\varphi\|_{\BMO^d(Q)}\le\ve,$ thus by Theorem~\ref{ta}, \eqref{i2} holds.

To prove sharpness, consider again the function $\varphi_*$ from~\eqref{i3} and extend it periodically to all of $\rn,$ by replicating it on every dyadic cube of measure 1. As before, for any $a\in\mathbb{R},$ let $\varphi_a=\ve\varphi_*+a.$ Then $\|\varphi_a\|_{\BMO^d(\rn)}=\ve\|\varphi_*\|_{\BMO^d((0,1)^n)}=\ve.$ Moreover, if $\ve<\ve_0^d(n),$ then inequality~\eqref{i2} becomes equality for  $Q=(0,1)^n$ and $\varphi=\varphi_a.$ On the other hand, if $\ve\ge\ve_0^d(n),$ we have $\av{e^{\varphi_a}}Q=\infty.$ 
\end{proof}

\section{$1$-oscillations vs. $2$-oscillations and equivalence of BMO norms}
\label{1-2}

In this section, we illustrate how one can obtain {\it lower} bounds for integral functionals on BMO on trees using $\alpha$-convex functions. A representative example is supplied by the task of sharply estimating the $p$-oscillation of a BMO function, $\Delta_{p,\mu,J}(\varphi),$ for $p<2,$ in terms of its $2$-oscillation $\Delta_{2,\mu,J}(\varphi)$ and the BMO norm. In the case of the continuous BMO on an interval, this problem was solved in~\cite{sv1}, where we built a family of locally convex functions on $\Oe$ that yielded sharp estimates of $L^p$-norms and, consequently, of $p$-oscillations. Here, we consider the simplest non-trivial case, $p=1,$ and construct the largest $\alpha$-convex analog of the corresponding function from~\cite{sv1}. This allows us to prove Theorems~\ref{t6} and~\ref{thb}. As in Section~\ref{JN}, there is an additional result of interest: Theorem~\ref{belb} gives the exact Bellman function for the $L^1$-norm of a $\BMO^d$ function in dimension $n.$

Take an $\ve>0$ and $\alpha\in(0,1/2]$ and write $\Oe=\Omega_0\cup\Omega_1,$ where 
\begin{align*}
\Omega_0&=\Big\{x:~\textstyle{\frac{(1+\alpha)\ve}{\sqrt\alpha}}\,|x_1|\le x_2\le x_1^2+\ve^2\Big\},\\
\Omega_1&=\Big\{x:~|x_1|\le\sqrt\alpha\ve,~x_1^2\le x_2\le \textstyle{\frac{(1+\alpha)\ve}{\sqrt\alpha}}\,|x_1|\Big\}\cup\Big\{x:~|x_1|>\sqrt\alpha\ve,~x_1^2\le x_2\le x_1^2+\ve^2\Big\}.
\end{align*}
Let
\eq[b]{
b_{\alpha,\ve}(x)=
\begin{cases}
\frac{\sqrt\alpha}{(1+\alpha)\ve}\,x_2, & x\in\Omega_0,\\
|x_1|, & x\in\Omega_1.
\end{cases}
}
We will show that $b_{\alpha,\ve}$ is $\alpha$-convex in $\Oe,$ which, in conjunction with Lemma~\ref{L1}, will yield the estimates of Theorem~\ref{thb}. We need one more definition: let
$$
\Omega_2=\Omega_0\cup\Big\{x\in\Omega_1:~x_2\le\textstyle{\frac{(1+\alpha)\ve}{\sqrt\alpha}}\,|x_1|\Big\}.
$$
Observe that if $x\in\Omega_2,$ then $b_{\alpha,\ve}(x)=\max\big\{|x_1|,\frac{\sqrt\alpha}{(1+\alpha)\ve}\,x_2\big\};$ if $x\in\Oe\setminus \Omega_2,$ then $|x_1|>\ve/\sqrt\alpha$ and $x_2> (1+\alpha)\ve^2/\alpha.$
\begin{lemma}
\label{L10}
The function $b_{\alpha,\ve}$ is $\alpha$-convex in $\Oe.$
\end{lemma}
\begin{proof}
Write $b$ for $b_{\alpha,\ve}.$ Because $b$ is piecewise linear, we can easily verify part~\eqref{602} of Definition~\ref{def} directly, without using Lemma~\ref{L2}.

Take any $\beta\in[\alpha,1/2]$ and any three points $P,Q,S\in\Oe$ such that $P=(1-\beta)S+\beta Q.$ If $P\in\Omega_1,$ then
$$
b(P)-(1-\beta)\,b(S)-\beta\,b(Q)\le |p_1|-(1-\beta)\,|s_1|-\beta\,|q_1|\le0,
$$ 
because $b(x)\ge|x_1|$ in $\Oe$ and $|x_1|$ is a convex function on the plane. 

If $P\in\Omega_0,$ then $|p_1|\le\sqrt\alpha\,\ve$ and $p_2\le(1+\alpha)\ve^2.$ Thus, 
$$
s_2=\frac{p_2-\beta q_2}{1-\beta}\le \frac{1+\alpha}{1-\beta}\,\ve^2\le \frac{1+\alpha}{\alpha}\,\ve^2,
$$
since $1-\beta\ge1/2\ge\alpha.$ Therefore, $S\in\Omega_2.$ On the other hand,
$$
(1-\beta)s_2=p_2-\beta q_2\le p_2-\alpha q_2\le(1+\alpha)\ve^2-\alpha q_2.
$$
Thus, $q_2\le \frac{(1+\alpha)\ve^2}\alpha,$ which means that $Q\in\Omega_2.$ Finally, since 
$b(x)\ge \frac{\sqrt\alpha}{(1+\alpha)\ve}\,x_2$ in $\Omega_2,$ we have
$$
b(P)-(1-\beta)\,b(S)-\beta\,b(Q)\le\textstyle{\frac{\sqrt\alpha}{(1+\alpha)\ve}}\,(p_2-(1-\beta)\,s_2-\beta\,q_2)=0.
$$
This completes the proof.
\end{proof}

We can now prove Theorem~\ref{thb}.
\begin{proof}[Proof of Theorem~\ref{thb}]
As in the proof of Theorem~\ref{ta}, take any $\varphi\in\BMO_\ve(\T)$ and for $N>0$ let $\varphi_N=\sum_{J\in\T_N}\avm{\varphi}J\chi^{}_J.$ Since $b_{\alpha,\ve}$ is $\alpha$-convex in $\Oe,$ and $b_{\alpha,\ve}(x)=|x_1|$ on $\Gamma_0,$ Lemma~\ref{L1} gives
$$
b_{\alpha,\ve}\big(\avm{\varphi^{}_N}X,\avm{\varphi_N^2}X\big)\le \frac1{\mu(X)}\,\int_X b_{\alpha,\ve}(\varphi^{}_N,\varphi_N^2)\,d\mu=\frac1{\mu(X)}\,\int_X |\varphi^{}_N|\,d\mu.
$$
Note that $|\varphi_N|\le |\varphi|_N,$ thus the right-hand side does not exceed $\avm{|\varphi|^{}_N}X=\avm{|\varphi|}X.$ Moreover, as in Lemma~\ref{L1}, the left-hand side converges to $b_{\alpha,\ve}\big(\avm{\varphi}X,\avm{\varphi^2}X\big)$ as $N\to\infty.$ Thus,
\eq[ov]{
b_{\alpha,\ve}\big(\avm{\varphi}X,\avm{\varphi^2}X\big)\le\avm{|\varphi|}X.
}
Recall definition~\eqref{b} of $b_{\alpha,\ve}.$ Replacing $\varphi$ with $\varphi-\avm{\varphi}X$ (which does not affect the $\BMO$ norm of $\varphi$), we obtain
$$
b_{\alpha,\ve}\big(0,\avm{\varphi^2}X-\avm{\varphi}X^2\big)=\frac{\sqrt\alpha}{(1+\alpha)\ve}\,\Delta_{2,\mu,X}(\varphi)\le\Delta_{1,\mu,X}(\varphi),
$$
which proves~\eqref{mest11} with $J=X.$ As before, by considering the tree $\mathcal{T}(J)$ in place of $\mathcal{T}$ we get the result for an arbitrary $J\in\mathcal{T}.$
Thus,
\eq[11]{
\frac{\sqrt\alpha}{(1+\alpha)\ve}\,\Delta_{2,\mu,J}(\varphi)\le\Delta_{1,\mu,J}(\varphi)\le \|\varphi\|_{\BMO^1(\T)}.
}
To prove~\eqref{mest22}, take any $\varphi\in\BMO(\T)$ and let $\ve=\|\varphi\|_{\BMO(\T)}.$ If $\ve=0,$~\eqref{mest22} holds automatically; if $\ve>0,$ we have a sequence $\{J_k\}$ of elements of $\T$ such that 
$\Delta_{2,\mu,J_k}(\varphi)\to\ve^2$ as $k\to\infty.$ Replacing $J$ with $J_k$ in~\eqref{11} and taking the limit completes the proof.
\end{proof}

In the dyadic case, we again state the corresponding result using a Bellman function. In contrast with the John--Nirenberg inequality, we are proving a sharp {\it lower} estimate, thus our Bellman function involves infimum instead of supremum. 

Take a dyadic cube $Q\subset\rn$ and for every $x\in\Oe,$ let
$$
\bel{b}^n_{\ve}(x)=\inf_{\varphi\in \BMO^d_{\ve}(Q)}\left\{\av{|\varphi|}Q:~
\av{\varphi}I=x_1, \av{\varphi^2}I=x_2\right\}.
$$
\begin{theorem}
\label{belb}
For any $\ve>0,$
$$
\bel{b}^n_\ve=b_{2^{-n},\ve}.
$$
\end{theorem}
\begin{proof}
The proof is similar to that of Theorem~\ref{thB}. Write $\bel{b}$ for $\bel{b}^n_\ve$ and $b$ for $b_{2^{-n},\ve}.$ Take any $x\in\Oe$ and any $\varphi\in\BMO^d_\ve(Q)$ such that $(\av{\varphi}Q,\av{\varphi^2}Q)=x.$ As shown in the proof of Theorem~\ref{thb} (equation~\eqref{ov} with $X=Q$ and $\mu$ the Lebesgue measure),
$$
b(x)\le\av{|\varphi|}Q.
$$
Taking the infimum over all such $\varphi,$ we conclude that $b\le\bel{b}$ in $\Oe.$

To prove the converse, we again let $Q=(0,1)^n$ and use the function~$\varphi_*$ from~\eqref{i3}. For all $a$ such that $|a|\ge 2^{-n/2}\ve,$ let $\varphi_a=\ve\varphi_*+ a.$ Recall that
$\|\varphi_a\|_{\BMO^d(Q)}=\ve$ and that $(\av{\varphi_a}Q,\av{\varphi_a^2}Q)=(a,a^2+\ve^2).$ Furthermore, $\varphi_a$ is almost everywhere positive for $a\ge 2^{-n/2}\ve$ and almost everywhere negative for $a\le -2^{-n/2}\ve,$ thus 
$$
\bel{b}(a,a^2+\ve^2)\le\av{|\varphi_a|}Q=|\av{\varphi_a}Q|=|a|=b(a,a^2+\ve^2).
$$
Hence, $\bel{b}=b$ on $\Gamma_\ve\cap \Omega_1.$ On the other hand, by considering the constant function $\psi_u=u$ corresponding to a point $(u,u^2)\in\Gamma_0$ we conclude that $\bel{b}=b$ on $\Gamma_0.$

Take any $x\in\Omega_1$ and let $\ell_x$ be any line that passes through $x,$ a point $U$ on $\Gamma_0,$ and a point $V$ on $\Gamma_\ve,$ and such that the segment $[U,V]\subset \Omega_1$ (there are infinitely many such lines for each $x$). Write $x=(1-\gamma)U+\gamma V$ for some $\gamma\in[0,1].$ Note that $b$ is linear along $[U,V],$ while $\bel{b}$ is a locally convex function directly from its definition. Therefore,
\eq[arg]{
\bel{b}(x)\le(1-\gamma)\,\bel{b}(U)+\gamma\,\bel{b}(V)=(1-\gamma)\,b(U)+\gamma\,b(V)=b(x). 
}
Thus, $\bel{b}=b$ in $\Omega_1.$ 

Now, take any $x\in\Omega_0$ and write $x$ a convex combination of the three ``corner'' points of $\Omega_0,$ $(0,0),$ $(-2^{-n/2}\ve,(1+2^{-n})\ve^2),$ and $(2^{-n/2}\ve,(1+2^{-n})\ve^2).$ Since at these three points $\bel{b}=b,$ $b$ is linear in $\Omega_0,$ and $\bel{b}$ is locally convex, arguing as in~\eqref{arg} proves that 
$\bel{b}\le b$ in $\Omega_0.$ Putting everything together, we conclude that $\bel{b}=b$ in $\Oe.$ 
\end{proof}

\begin{proof}[Proof of Theorem~\ref{t6}]
Arguing as in the proof of Theorem~\ref{t1}, we obtain~\eqref{i6} and~\eqref{i7} from Theorem~\ref{thb}. To prove sharpness of~\eqref{i6}, we do not present an explicit function as we did in Theorem~\ref{t1}. Instead, take any $\ve>0$ and any $Q\in\mathcal{D}.$ Theorem~\ref{belb} says that $\bel{b}_\ve^n(0,\ve^2)=
\frac{2^{n/2}}{2^n+1}\,\ve.$ Therefore, there exists a sequence of functions $\{\varphi_k\}$ from $\BMO^d(Q)$ such that $\av{\varphi_k}Q=0,$ $\|\varphi_k\|_{\BMO^d(Q)}=\av{\varphi^2_k}Q^{1/2}=\ve,$ and
$$
\lim_{k\to\infty}\av{|\varphi_k-\av{\varphi_k}Q|}Q=\lim_{k\to\infty}\av{|\varphi_k|}Q=\frac{2^{n/2}}{2^n+1}\,\ve.
$$
Extending each $\varphi_k$ periodically to all of $\rn$ we get $\|\varphi_k\|_{\BMO^d(\rn)}=\ve,$ which completes the proof.
\end{proof}

\section{$\alpha$-martingales and  continuous BMO}

\label{martingales}
Any metric measure space has a canonical BMO defined on balls; we call this the continuous BMO. If the space is equipped with an $\alpha$-tree $\mathcal{T},$ one also has $\BMO(\mathcal{T}).$  It is natural to ask whether our theory for BMO on trees can be used to obtain sharp inequalities for the continuous BMO. If the tree in question is fixed from the outset, the answer is no; however, if one allows each BMO function to generate its own tree and, thus, a martingale, then the answer is yes, provided all such martingales are regular in a particular uniform sense.

To focus the discussion, consider the classical BMO on Euclidean cubes (the continuous $\BMO(\mathbb{R}^n)$ in the $L^\infty$-metric). Take a cube $Q$ in $\mathbb{R}^n$ and, as in~\eqref{1}, let
$$
\BMO(Q)=\big\{\varphi\in L^1(Q)\colon\|\varphi\|_{\BMO(Q)}:=\sup_{\text{cube}~J\subset Q}\big(\Delta_{2,J}\big)^{1/2}<\infty\big\}.
$$

We first give a geometric definition of an $\alpha$-concave function that works for the full range $\alpha\in(0,1];$ $\alpha$-convex functions are defined symmetrically. The reader can easily verify that this definition is equivalent to Definition~\ref{def} for $\alpha\in(0,1/2].$ 
\begin{definition}
\label{def1}
Take $\ve>0$ and $\alpha\in\big(0,1].$ Let $P$ and $R$ be two distinct points from $\Oe.$ The segment $[P,R]$ is called $\alpha$-good, if the length of the portion of $[P,R]$ that lies outside $\Oe$ does not exceed $(1-\alpha)\len{P}{R}.$ A function $B$ on $\Oe$ is called $\alpha$-concave, if for each $\alpha$-good segment $[P,R]$ and each $Q\in[P,R]\cap\Oe$ we have
\eq[6016]{
B(Q)\ge \frac{\len{Q}{R}}{\len{P}{R}}\,B(P)+\frac{\len{P}{Q}}{\len{P}{R}}\,B(R).
}
\end{definition}
Observe that the larger the $\alpha,$ the weaker the condition of $\alpha$-concavity, and $1$-concavity is the same as local concavity.

We now introduce special trees and martingales adapted to those trees.

\begin{definition}
\label{def2}
Let $Q$ be a cube in $\mathbb{R}^n.$ A disjoint collection $\mathcal{T}$ of measurable subsets of $Q$ is called a binary tree on $Q$ if 
\ben
\item
For each $J\in\mathcal{T},$ $J=J^-\cup J^+,$ with $J^-,J^+\in\mathcal{T};$
\item
$\mathcal{T}=\cup\mathcal{T}_k,$ where $\mathcal{T}_0=\{Q\}$ and $\mathcal{T}_{k+1}=\cup_{J\in\mathcal{T}_k}\{J^-,J^+\};$
\item
For all $J\in\mathcal{T},$ $|J|>0,$ and $\lim_{k\to\infty}\max_{J\in\mathcal{T}_k}|J|=0.$
\een
\end{definition}

\begin{definition}
\label{def3}
Let $\mathcal{T}$ be a binary tree on a cube $Q,$ $\alpha\in(0,1],$ and $\ve>0.$ Let $\{x_J\}_{J\in\mathcal{T}}$ be a collection of points from $\Oe$ such that for any $J,$ $|J|\,x_J=|J^-|\,x_{J^-}+|J^+|\,x_{J^+}.$ Let $\{f_k\}$ be the corresponding sequence of $\Oe$-valued simple functions on $Q$ given by
$$
f_k=\sum_{J\in\mathcal{T}_k}x^{}_J\,\chi^{}_J.
$$
Assume that for each $J\in\mathcal{T},$ either $x_{J^-}=x_{J^+}$ or the segment $[x_{J^-},x_{J^+}]$ is $\alpha$-good. Then $f_k$ is called an $(\alpha,\ve)$-martingale adapted to the tree $\mathcal{T}.$

Furthermore, if $\varphi\in L^2(Q)$ is such that for all $k$ and all $J\in\mathcal{T}_k,$ $(\av{\varphi}J,\av{\varphi^2}J)=f_k|_J,$ then we say that $\varphi$ generates $f_k.$ Lastly, an $\alpha$-martingale is an $(\alpha,\|\varphi\|_{\BMO(Q)})$-martingale generated by a function $\varphi\in\BMO(Q).$
\end{definition}
Observe that by Lemma~\ref{L0}, as used in the proof of Lemma~\ref{L1}, any function $\varphi\in\BMO^d(Q)$ generates a $2^{-n}$-martingale.

Combining the notions of $\alpha$-concave functions and $\alpha$-martingales gives the following general result, which can be viewed as a special form of Jensen's inequality. The first part is contained in Lemma~\ref{L1} and the second part follows from Fatou's Lemma, as in the proof of Theorem~\ref{ta}.
\begin{lemma}
\label{L66}
Let $Q$ be a cube in $\mathbb{R}^n,$ $\ve>0,$ and $\alpha\in(0,1].$ Let $B$ be an $\alpha$-concave function on $\Oe.$ 

If $f_k$ is an $(\alpha,\ve)$-martingale on $Q,$ then for any $k\ge0$
$$
B(f_0)\ge \frac1{|Q|}\int_QB(f_k).
$$
Consequently, if $\varphi\in L^2(Q)$ generates an $(\alpha,\ve)$-martingale and $B$ is continuous on $\Oe,$ then
$$
B(\av{\varphi}Q,\av{\varphi^2}Q)\ge\frac1{|Q|} \int_Q B(\varphi,\varphi^2).
$$
\end{lemma}
Let us illustrate this lemma in the case of the John--Nirenberg inequality. To state our result, we need two new parameters. Take a cube $Q.$ Let
$$
\ve_0(n)=\sup\Big\{\ve>0:~\forall\varphi\in\BMO(Q),~\sup_{\text{cube}~J\subset Q}\av{e^{\ve(\varphi-\av{\varphi}{\scriptscriptstyle J})/
\|\varphi\|_{\BMO(Q)}}}J<\infty\Big\}
$$
and
$$
\alpha_0(n)=\sup\{\alpha\in(0,1]:~\text{every}~ \varphi\in\BMO(Q)  \text{~generates an $\alpha$-martingale on $Q$}\}.
$$
(Of course, neither $\ve_0(n)$ nor $\alpha_0(n)$ actually depends on $Q.$) Observe that both suprema are taken over non-empty sets: since every 
$\varphi\in\BMO(Q)$ is also in $\BMO^d(Q)$ with at most the same norm, we have $\ve_0(n)\ge\ve_0^d(n)=\frac{2^{n/2}}{2^n-1}\,n\,\log 2,$ and $\alpha_0(n)\ge 2^{-n}.$

Recall the function $B_{\alpha,\ve}$ given by~\eqref{ba}. Lemma~\ref{lba} asserts that $B_{\alpha,\ve}$ is $\alpha$-concave on $\Oe$ for $\alpha\in(0,1/2].$ In fact, the same proof also works for $\alpha\in(1/2,1].$  Using this function in Lemma~\ref{L66} gives the following corollary.
\begin{corollary}
\eq[last]{
\ve_0(n)\ge \frac{\sqrt{\alpha_0(n)}}{1-\alpha_0(n)}\log\Big(\frac1{\alpha_0(n)}\Big)
}
\end{corollary}
A natural conjecture is that~\eqref{last} holds with equality. It does when $n=1,$ which is the only case where we currently know $\alpha_0(n).$ To explain: the main geometric result that underpins the now-well-developed Bellman theory for BMO in dimension 1 is Lemma~4c of~\cite{sv}. For an interval $I,$ any $\delta\in(0,1),$ and any $\varphi\in\BMO(I),$ that lemma gives an explicit construction of a $(1-\delta)$-martingale generated by $\varphi.$ Therefore, $\alpha_0(1)=1$ and so $\ve_0(1)\ge1;$ a simple logarithmic example shows that $\ve_0(1)=1.$  

In higher dimensions we do not yet know either $\alpha_0(n)$ or $\ve_0(n).$  However, in our view inequality~\eqref{last} expresses the right idea: to find $\ve_0(n),$ first understand the nature of $\alpha$-martingales generated by BMO functions. The trees that correspond to the best such martingales (the ones with the largest $\alpha$) may not be given by the procedure of Lemma~\ref{L1}, where on each step we split off a cube from a larger set. Indeed, consider the function on the unit square $Q=(0,1)^2$ that equals zero on two of the four quarters of $Q$ and $\pm\sqrt2$ on the other two. This function has BMO norm $1.$ If one builds a binary martingale by first splitting off one quarter-square of the four, then one of the remaining three, then splitting the last two, one obtains a $1/2$-martingale. If, however, one first splits $Q$ into two ``halves,'' each containing one quarter on which the function is zero, one obtains a $1$-martingale.


\begin{thebibliography}{99}
\bibitem{b}
D. L. Burkholder. Boundary value problems and sharp inequalities for 
martingale transforms. {\it Annals of Probability}, Vol.~12 (1984), 
No.~3, pp.~647--702.
\bibitem{christ}
M. Christ. A $T(b)$ theorem with remarks on analytic capacity and the Cauchy integral. {\it Colloq. Math.}, Vol.~60/61 (1990), No.~2, pp.~601-628.
\bibitem{jn}
F. John, L. Nirenberg. On functions of bounded mean oscillation. {\it Comm. Pure Appl. Math.}, Vol.~14 (1961), pp.~415-426.
\bibitem{iosvz1}
P.~Ivanishvili, N.~Osipov, D.~ Stolyarov, V.~Vasyunin, P.~Zatitskiy, On Bellman function for extremal problems in BMO. {\it C. R. Math. Acad. Sci. Paris} 350 (2012), No.~11-12, pp.~561-564.
\bibitem{iosvz2}
P. Ivanishvili, N. Osipov, D. Stolyarov, V.Vasyunin, P. Zatitskiy. Bellman function for
extremal problems in BMO. To appear in {\it Transactions of the AMS}, pp.~1-91, \url{http://arxiv.org/pdf/1205.7018.pdf}\bibitem{melas}
A. Melas. Dyadic-like maximal operators on $L\log L$ functions. {\it J. Funct. Anal.}, Vol.~257 (2009), No.~6, pp.~1631-1654. 
\bibitem{mns}
A.~Melas, E.~Nikolidakis, T.~Stavropoulos. Sharp local lower $L^p$-bounds for dyadic-like maximal operators.
{\it Proc. Amer. Math. Soc.}, Vol.~141 (2013), No.~9, pp.~3171-3181.
\bibitem{nt}
F. Nazarov, S. Treil. The hunt for Bellman function: applications
to estimates of singular integral operators and to other classical
problems in harmonic analysis. (Russian){\it Algebra i Analiz}, Vol.~8
(1996), No.~5, pp.~32-162; English translation in {\it St. Petersburg 
Math. J.}, Vol.~8 (1997), No.~5, pp.~721-824.
\bibitem{ntv1}
F. Nazarov, S. Treil, A. Volberg. The Bellman functions and two-weight 
inequalities for Haar multipliers. 1995, Preprint, MSU, pp.~1-25.
\bibitem{ntv2}
F. Nazarov, S. Treil, A. Volberg. The Bellman functions and two-weight 
inequalities for Haar multipliers. {\it Journal of the American 
Mathematical Society}, Vol.~12 (1999), No.~4, pp.~909-928.
\bibitem{sv}
L. Slavin, V. Vasyunin. Sharp results in the integral-form 
John--Nirenberg inequality. Trans. Amer. Math. Soc., Vol.~363 (2011), No.~8, pp.~4135-4169.
\bibitem{sv1}
L. Slavin, V. Vasyunin. Sharp $L^p$ estimates on BMO norms. {\it Indiana Univ. Math. J.}, Vol.~61 (2012), No.~3, pp.~1051-1110.
\bibitem{vv}
V. Vasyunin, A. Volberg, Sharp constants in the classical weak form of the John--Nirenberg inequality. {\it Proc. Lond. Math. Soc. (3)}, Vol.~108 (2014), No.~6, pp.~1417-1434.
\end{thebibliography}
\end{document}